\pdfoutput=1
\documentclass[11pt]{article}
\usepackage[letterpaper, margin=1.2in, top=1in, bottom=1.2in]{geometry}
\usepackage[T1]{fontenc}
\usepackage[utf8]{inputenc}
\usepackage{amsmath,amssymb,amsthm,mathtools}
\usepackage{aliascnt}
\usepackage{enumitem}
\usepackage{hyperref}
\usepackage[nameinlink,capitalise]{cleveref}
\usepackage{mathrsfs}
\usepackage{setspace}
\usepackage{xcolor}
\usepackage{comment}

\numberwithin{equation}{section}

\AtBeginDocument{%
  \setlength{\abovedisplayskip}{9pt plus 3pt minus 2pt}%
  \setlength{\belowdisplayskip}{9pt plus 3pt minus 2pt}%
  \setlength{\abovedisplayshortskip}{6pt plus 2pt minus 1pt}%
  \setlength{\belowdisplayshortskip}{7pt plus 2pt minus 1pt}%
}

\hypersetup{colorlinks=true,linkcolor=blue,citecolor=blue,urlcolor=blue}

\makeatletter
\def\thm@space@setup{%
  \thm@preskip=8pt plus 2pt minus 2pt
  \thm@postskip=8pt plus 2pt minus 2pt
}
\makeatother

\newtheorem{theorem}{Theorem}[section]
\newaliascnt{proposition}{theorem}
\newtheorem{proposition}[proposition]{Proposition}
\aliascntresetthe{proposition}
\newaliascnt{lemma}{theorem}
\newtheorem{lemma}[lemma]{Lemma}
\aliascntresetthe{lemma}
\newaliascnt{corollary}{theorem}
\newtheorem{corollary}[corollary]{Corollary}
\aliascntresetthe{corollary}
\newaliascnt{definition}{theorem}

\aliascntresetthe{definition}
\newaliascnt{question}{theorem}
\newtheorem{question}[question]{Question}
\aliascntresetthe{question}
\newaliascnt{example}{theorem}
\newtheorem{example}[example]{Example}
\aliascntresetthe{example}

\crefname{theorem}{Theorem}{Theorems}
\Crefname{theorem}{Theorem}{Theorems}
\crefname{proposition}{Proposition}{Propositions}
\Crefname{proposition}{Proposition}{Propositions}
\crefname{lemma}{Lemma}{Lemmas}
\Crefname{lemma}{Lemma}{Lemmas}
\crefname{corollary}{Corollary}{Corollaries}
\Crefname{corollary}{Corollary}{Corollaries}
\crefname{definition}{Definition}{Definitions}
\Crefname{definition}{Definition}{Definitions}
\crefname{remark}{Remark}{Remarks}
\Crefname{remark}{Remark}{Remarks}
\crefname{question}{Question}{Questions}
\Crefname{question}{Question}{Questions}
\crefname{example}{Example}{Examples}
\Crefname{example}{Example}{Examples}

\newtheorem{mainthminner}{Theorem}
\newenvironment{mainthm}[2][]{%
  \renewcommand{\themainthminner}{#2}%
  \begin{mainthminner}[#1]%
}{\end{mainthminner}}
\crefname{mainthminner}{Theorem}{Theorems}
\Crefname{mainthminner}{Theorem}{Theorems}

\newtheorem{mainexampleinner}{Example}

\crefname{mainexampleinner}{Example}{Examples}
\Crefname{mainexampleinner}{Example}{Examples}

\theoremstyle{remark}
\newaliascnt{remark}{theorem}
\newtheorem{remark}[remark]{Remark}
\aliascntresetthe{remark}

\newcommand{\N}{\mathbb N}

\newcommand{\Z}{\mathbb Z}
\newcommand{\Q}{\mathbb Q}
\newcommand{\R}{\mathbb R}
\newcommand{\T}{\mathbb T}

\newcommand{\floor}[1]{\left\lfloor #1\right\rfloor}
\newcommand{\ceil}[1]{\left\lceil #1\right\rceil}
\newcommand{\round}[1]{\left[#1\right]}
\newcommand{\normT}[1]{\left\|#1\right\|_{\mathbb T}}
\newcommand{\nablar}{\nabla_{\mathbb R}\text{-}\mathrm{span}}

\newcommand{\nablaz}{\nabla_{\mathbb Z}\text{-}\mathrm{span}}

\title{Bohr obstructions to recurrence along Hardy-field sequences}
\author{Kangbo Ouyang \and Sa\'ul Rodr\'iguez-Mart\'in \and Leiye Xu \and Shuhao Zhang}
\date{}

\begin{document}
\maketitle

\begin{abstract}
We construct Bohr obstructions to multiple recurrence along rounded Hardy-field sequences, showing that the real derivative-span criterion of Bergelson, Moreira, and Richter is essentially sharp and answering two of their questions. For $E\subseteq\mathbb N$ and $u:\mathbb N\to\mathbb Z$, set
$R_{u}(E):=\{n\in\mathbb N:E\cap(E-u(n))\neq\varnothing\}$. We prove that, if $f_1,\dots,f_k$ are functions of polynomial growth from a Hardy field and some real linear combination of $f_1,\dots,f_k$ and their derivatives has a nonzero
finite limit, then there exist $M\in\mathbb N$ and a basic Bohr set
$E\subseteq\mathbb N$ such that
$\bigcap_{i=1}^k R_{[Mf_i]}(E)$ is not thick.
In particular, for some
Bohr set $E$, the set
$R_{[t^{3/2}]}(E)\cap R_{[\sqrt{2}t^{3/2}+t]}(E)$ is
piecewise syndetic but not thick.
We also prove that, if for some $\lambda_1,\dots,\lambda_k\in\mathbb{R}$ we have
$$\inf_{x\geq 1}\left\|\sum_{i}\lambda_if_i(x)\right\|_{\mathbb{T}}>\frac{1}{2}
\sum_i|\lambda_i|,$$
then $\bigcap_i R_{[f_i]}(E)=\varnothing$ for some basic Bohr set $E$.

More generally, our results apply with $[\cdot]$ replaced by any rounding function $\rho:\mathbb{R}\to\mathbb{Z}$ satisfying $\sup_{x\in\mathbb{R}}|\rho(x)-x|<\infty$.

\end{abstract}

\bigskip
\begin{center}
\begin{minipage}{0.92\textwidth}
\noindent\textbf{2020 Mathematics Subject Classification.}
Primary 37A44; Secondary 37B20, 11B30.

\smallskip
\noindent\textbf{Keywords.}
Hardy-field sequences, multiple recurrence, Bohr sets, return-time sets,
jointly intersective polynomials, derivative span.
\end{minipage}
\end{center}
\bigskip

\section{Introduction and main results}

\subsection{Background and criteria of Bergelson--Moreira--Richter}

This paper studies recurrence along integer sequences obtained by rounding Hardy-field functions. For $E\subset\N$ we write
$$
\overline d(E)=\limsup_{N\to\infty}\frac{|E\cap\{1,\ldots,N\}|}{N},
\qquad
\underline d(E)=\liminf_{N\to\infty}\frac{|E\cap\{1,\ldots,N\}|}{N},
$$
and write $d(E)$ when these two quantities agree.

Our starting point is the phenomenon that sets of integers with positive density must contain rich additive structure. An example of this phenomenon is provided by Szemerédi’s theorem~\cite{Sze75} on arithmetic progressions (also proved ergodically by Furstenberg~\cite{Fur77}): if $E\subseteq\mathbb{N}$ has positive density and any $k\in\mathbb{N}$, then there exist $n,a\in\mathbb{N}$ such that 
\begin{equation*}
a,a+n,a+2n,\dots,a+kn\in E.
\end{equation*}
This result extends to some configurations of the form 
\begin{equation*}
a,a+p_1(n),a+p_2(n),\dots,a+p_k(n), 
\end{equation*}
for polynomials $p_1,\dots,p_k\in\mathbb{Z}[t]$. In this case, recurrence is governed by the following purely arithmetic obstruction. 
We say $p_1,\dots,p_k\in\Z[t]$ are
\emph{jointly intersective} if for every $m\in\N$ there is $n\in\N$ such that
$p_i(n)\equiv0\pmod m$ for all $i=1,\dots,k$. Joint intersectivity is a necessary condition for recurrence: if $p_1,\dots,p_k$ cannot vanish simultaneously modulo some $m\in\mathbb{N}$, then there cannot be 
$n,a\in\mathbb{N}$ such that 
$a,a+p_1(n),a+p_2(n),\dots,a+p_k(n)\in m\mathbb{N}$. The following theorem of
Bergelson, Leibman, and Lesigne shows that joint intersectivity is also
sufficient for recurrence.

\begin{proposition}[{\cite[Theorem~1.1]{BLL08}}]
For polynomials $p_1,\dots,p_k\in\Z[t]$, the following are equivalent:
\begin{enumerate}
    \item $p_1,\dots,p_k$ are jointly intersective.
    \item For every set $E\subset\N$ with $\overline d(E)>0$, there exist
    $a,n\in\N$ such that
    $$
    \{a,a+p_1(n),\dots,a+p_k(n)\}\subset E.
    $$
\end{enumerate}
\end{proposition}
Recurrence has also been studied for sequences obtained from rounding by Hardy-field functions (see \cite{Bos81}). Hardy fields contain a much broader class of functions than polynomials, including
$t^{3/2}$, $t\log t$, or linear combinations of functions of different growth.
Results of Frantzikinakis and Wierdl~\cite{FW09} and
Frantzikinakis~\cite{Fra10,Fra15} established recurrence and convergence for
large classes of Hardy-field sequences.  Bergelson, Moreira, and
Richter~\cite{BMR24} later obtained general criteria which make precise which
real-linear asymptotic relations among the functions can obstruct recurrence.

We first fix the discretization.  A map $\rho:\R\to\Z$ will be called a
\emph{rounding map} if
$$
K_\rho:=\sup_{x\in\R}|\rho(x)-x|<\infty.
$$
The standard examples are $\floor{x}$, $\ceil{x}$, and the nearest-integer
map, for which we use the convention $\round{x}=\lfloor x+0.5\rfloor$.  If $u:\N\to\Z$ and
$E\subset\N$, define
$$
R_u(E):=\{n\in\N:E\cap(E-u(n))\ne\varnothing\},
\qquad
E-r:=\{m\in\N:m+r\in E\}.
$$
Thus $R_u(E)$ is the set of times $n$ for which $u(n)$ occurs as a difference
of two elements of $E$.  For a real-valued function $f$ we write
$R_{\rho\circ f}(E)$ for the same set with $u(n)=\rho(f(n))$.

The first criterion from~\cite{BMR24} concerns when such return-time sets are
thick.  Recall that $A\subset\Z$ is \emph{thick} if for every $N\in\N$ there
exists $a\in\Z$ with $\{a+1,\dots,a+N\}\subset A$.  For   functions
$f_1,\ldots,f_k$, set
\begin{align*}
\nablar(f_1,\ldots,f_k)&=\operatorname{span}_{\R}\{f_i^{(m)}:1\le i\le k,\ m\ge0\},\\
\nablaz(f_1,\ldots,f_k)&=\left\{\sum_{i=1}^k\sum_{m=0}^{M}c_{i,m}f_i^{(m)}:M\ge0,\ c_{i,m}\in\Z\right\},\\
  \mathcal I_{\Z}(f_1,\dots,f_k) &=
  \Bigl\{\sum_{i=1}^k c_i f_i^{(m_i)}:
  c_i\in\Z,\ m_i\in\N\cup\{0\}\Bigr\}.
\end{align*}

\begin{theorem}[Bergelson--Moreira--Richter, {\cite[Corollary~A4]{BMR24}}]
\label{thm:BMR-thick}
Let $f_1,\ldots,f_k$ be functions of polynomial growth from a Hardy field, and
assume that
$$
\lim_{t\to\infty}|F(t)|\in\{0,\infty\}
\qquad\text{for every }F\in\nablar(f_1,\ldots,f_k).
$$
Then for every $E\subset\N$ with $\overline d(E)>0$ and every $\ell\in\N$,
there exist $a,n\in\N$ such that
$$
\{a\}\cup
\{a+[f_i(n+j)]:1\le i\le k,\ 0\le j\le\ell\}
\subset E.
$$
In particular, $\bigcap_{i=1}^kR_{[f_i]}(E)$ is thick. 
\end{theorem}
Since this sufficient condition is formulated with real coefficients but the conclusion concerns integer-valued shifts, Bergelson, Moreira, and Richter asked whether integer-coefficient derivative data suffice.

\begin{question}[{\cite[Question~6.6]{BMR24}}]\label{ques:Q66}
Does the conclusion of \Cref{thm:BMR-thick} remain true if one only assumes $\lim_{t\to\infty}|F(t)|\in\{0,\infty\}$ for   every
  \(F\in\mathcal I_{\Z}(f_1,\dots,f_k)\)? In particular, is the conclusion of \Cref{thm:BMR-thick} true for $f_1(t)=t^{3/2}$ and $f_2(t)=\alpha t^{3/2}+t$, where $\alpha\in\mathbb{R}\setminus\mathbb{Q}$?
  \end{question}
  
Their second question concerns polynomial shadows. Let $\mathcal{P}_{\mathbb{R}}(f_1,\ldots,f_k)$ be the set of all $p\in\R[t]$ such that $\lim_{x\to\infty}f-p=0$ for some $f\in\operatorname{span}_{\R}(f_1,\ldots,f_k)$. 

\begin{theorem}[Bergelson--Moreira--Richter, part of {\cite[Theorem~A]{BMR24}}]
\label{thm:BMR-A}
Let $f_1,\ldots,f_k$ be of polynomial growth from a Hardy field.  Suppose that there are jointly intersective polynomials $q_1,\ldots,q_\ell\in\Z[t]$ with
\[
\mathcal{P}_{\mathbb{R}}(f_1,\ldots,f_k)\subset  \operatorname{span}_{\R}(q_1,\ldots,q_\ell).
\]
Then every $E\subset\N$ of positive upper density contains $\{a,a+\round{f_1(n)},\ldots,a+\round{f_k(n)}\}$ for some $a,n\in\N$.
\end{theorem}

Let $\mathcal P_{\Z}(f_1,\ldots,f_k)$ be the set of all $q\in\Z[t]$ such that $|f-q|\to0$ for some non-zero $f\in\operatorname{span}_{\R}(f_1,\ldots,f_k)$, and say that an infinite collection of polynomials $\Q\subseteq\mathbb{Z}[t]$ is jointly intersective if for all $m\in\mathbb{N}$ there is $n\in\mathbb{N}$ such that $q(n)$ is a multiple of $m$ for all $q\in\Q$.

\begin{question}[{\cite[Question~6.5]{BMR24}}]\label{ques:Q65}
Can the strong polynomial-span assumption in \Cref{thm:BMR-A} be replaced by the weaker assumption that $\mathcal P_{\Z}(f_1,\ldots,f_k)$ is jointly intersective?
\end{question}

\begin{remark}
It is worth noting that the bounded error condition of the rounding function, $\rho(t)=t+O(1)$, is insufficient to guarantee non-empty multiple returns without further arithmetic conditions on the range of $\rho$. Indeed, if we let $\rho(t)=2\lfloor t/2\rfloor+1$, then for any \(f:\mathbb{N}\to\mathbb{R}\) we will have \(R_{\rho\circ f}(2\mathbb{N})=\varnothing\), because for all \(n\in\mathbb{N}\) the number \(\rho(f(n))\) is odd. However, the Bohr-obstruction results in this article hold true for arbitrary rounding functions.
\end{remark}

\subsection{Main results}

Our first two main results show that the derivative-span criterion of \Cref{thm:BMR-thick} is close to optimal. In the derivative-only case no homothety is needed.

In this article, for $x\in\mathbb{R}$ we use the standard notation
$\normT{x}=\min_{m\in\Z}|x-m|$ and view it as the distance to $0$ in
$\T=\R/\Z$.
Recall that a set $B\subset\N$ is said to be a Bohr neighborhood of $0$ if it contains some basic Bohr set, by which we mean a set of the form
\[
B(v_1,\dots,v_d;\varepsilon)
:=
\{n\in\N:\normT{nv_j}<\varepsilon \text{ for }1\le j\le d\},
\]
for some $v_1,\dots,v_d\in\mathbb{R}$ and $\varepsilon>0$.
It is well-known that $B(v_1,\dots,v_d;\varepsilon)$ has positive natural density (see \Cref{lem:bohr-positive-density}), although we include a proof for convenience of the reader.

\begin{mainthm}[Derivative-only obstruction]{A}\label{thm:A}
Let $f_1,\ldots,f_k$ be functions of polynomial growth from a Hardy field, and let $\rho:\R\to\Z$ satisfy $\sup_x|\rho(x)-x|<\infty$.  If some $g\in\nablar(f_1',\ldots,f_k')$ satisfies $\lim_{t\to\infty}|g(t)|\in(0,\infty)$, then there exists a basic Bohr set \(E\subset\N\) such that
$R_{\rho\circ f_1}(E)\cap\cdots\cap R_{\rho\circ f_k}(E)$ is not thick.
\end{mainthm}

For the full derivative span, the obstruction holds after a positive homothety.

\begin{mainthm}[Full-span obstruction up to homothety]{B}\label{thm:B}
Let $f_1,\ldots,f_k$ be functions of polynomial growth from a Hardy field, and let $\rho:\R\to\Z$ satisfy $\sup_x|\rho(x)-x|<\infty$.  If some $g\in\nablar(f_1,\ldots,f_k)$ satisfies $\lim_{t\to\infty}|g(t)|\in(0,\infty)$, then there exist $M\in\mathbb{N}$ and a basic Bohr set \(E\subset\N\) such that $R_{\rho\circ(Mf_1)}(E)\cap\cdots\cap R_{\rho\circ(Mf_k)}(E)$ is not thick.
\end{mainthm}

\begin{remark}
\label{4rew9dos9oewpdsl}
The constant \(M\) in \Cref{thm:B} is necessary, not merely a by-product of the proof. For example, let \(f_1(x)=\log x\) and \(f_2(x)=\log x+0.1\). If \(E\subset\N\) is infinite, then \(E-E\) contains arbitrarily large positive integers. For every \(d\in(E-E)\cap\N\), all \(m\in\N\) with \(d\le \log m<d+0.9\) belong to \(R_{\lfloor f_1\rfloor}(E)\cap R_{\lfloor f_2\rfloor}(E)\). These intervals contain arbitrarily long blocks of consecutive integers, so the intersection is thick.

On the other hand, after the homothety \(M=10\), the set \(R_{\lfloor 10f_1\rfloor}(2\N)\cap R_{\lfloor 10f_2\rfloor}(2\N)\) is empty. Indeed, a return time for \(2\N\) must be even, whereas \(\lfloor 10\log n+1\rfloor=\lfloor 10\log n\rfloor+1\) for every \(n\in\N\), so the two shifts cannot both be return times for \(2\N\).
\end{remark}
We further give a sufficient criterion for sets of multiple Hardy returns to be empty, which will allow us to answer another question of \cite{BMR24}.

\begin{mainthm}[Empty returns]{C}\label{Thm:C}
Let $\rho:\R\to\Z$ be a rounding map, and suppose some functions $f_1,\dots,f_k:\mathbb{N}\to\mathbb{R}$ satisfy, for some $N\in\mathbb{N}$ and $\lambda_1,\dots,\lambda_k\in\mathbb{R}$, that
\begin{equation}
\label{5te9rpdfol}
\inf_{x\geq N}\left\|\sum_{i}\lambda_if_i(x)\right\|_{\mathbb{T}}> K_\rho\sum_i|\lambda_i|.
\end{equation}
Then there is a basic Bohr set \(E\subset\N\) such that \(\cap_{i=1}^k R_{\rho\circ f_i}(E)\cap [N,\infty)=\varnothing\).
\end{mainthm}

\begin{remark}
Suppose some functions $f_1,\dots,f_k$ in a Hardy field satisfy the hypotheses of \Cref{Thm:C}. This means that $f(t)=\sum_{i}\lambda_if_i(t)$ is a Hardy field function satisfying 
$\liminf_{n\to\infty}\|f(n)\|_{\mathbb{T}}>0$, which can only happen if $f$ is of the form
\begin{equation*}
f(t)=p(t)+\kappa+o(1),
\end{equation*}
for some $p\in\mathbb{Q}[x]$ and $\kappa\in\mathbb{R}$. Otherwise, we would have $\lim_{t\to\infty}|f(t)-p(t)|=\infty$ for all $p\in\mathbb{Q}[x]$, which implies by \cite[Theorem 1.4]{Bos94} that $f(n)$ is dense mod $1$, a contradiction. 
\end{remark}

We have a simpler version of \Cref{Thm:C} if we allow a homothety constant $M>0$:
\begin{corollary}
\label{Thm:CHomothety}
Let $\rho:\R\to\Z$ be a rounding map, and suppose some functions $f_1,\dots,f_k:\mathbb{N}\to\mathbb{R}$ satisfy, for some $N\in\mathbb{N}$ and $\lambda_1,\dots,\lambda_k\in\mathbb{R}$, that
\begin{equation*}
\inf_{x\geq N}\left\|\sum_{i}\lambda_if_i(x)\right\|_{\mathbb{T}}>0.
\end{equation*}
Then there are $M>0$ and a basic Bohr set \(E\subset\N\) such that \(\cap_{i=1}^k R_{\rho\circ(Mf_i)}(E)\cap [N,\infty)=\varnothing\).
\end{corollary}

We now use the theorems above to answer the two questions of \cite{BMR24}. Recall that a subset of $\mathbb{N}$ is \emph{syndetic} if it
has bounded gaps, and \emph{piecewise syndetic} if it is the intersection of a
thick set and a syndetic set.

\begin{example}[Negative answer to \Cref{ques:Q66}]\label{exa:C}
Let $\lambda\in\R\setminus\Q$ and put $h_1(t)=t^{3/2}$, $h_2(t)=\lambda t^{3/2}+t$.  Then every $F\in\nablaz(h_1,h_2)$ satisfies $\lim_{t\to\infty}|F(t)|\in\{0,\infty\}$. However, there exists $E\subset\N$ of positive natural density such that $R_{\lfloor h_1\rfloor}(E)\cap R_{\lfloor h_2\rfloor}(E)$ is piecewise syndetic but not thick. The same holds with $\floor{\cdot}$ replaced by $\ceil{\cdot}$ or $\round{\cdot}$.
\end{example}

\begin{example}[Empty return-time obstruction]\label{exp:D}
Let $\lambda\in\R\setminus\Q$, let $\xi\in\R\setminus\Z$, and choose $Q\in\N$ with $Q\normT{\xi}>1+|\lambda|$.  Put
\[
g_1(t)=t^{3/2},\qquad g_2(t)=\lambda t^{3/2}+Q(t+\xi).
\]
Then every $F\in\nablaz(g_1,g_2)$ satisfies $\lim_{t\to\infty}|F(t)|\in\{0,\infty\}$.  Nevertheless, there exists a basic Bohr set \(E\subset\N\) such that
\[
R_{\lfloor g_1\rfloor}(E)\cap R_{\lfloor g_2\rfloor}(E)=\varnothing.
\]
The same holds with $\floor{\cdot}$ replaced by $\ceil{\cdot}$ or $\round{\cdot}$.
\end{example}
For \(\xi=\sqrt2\), the pair \((g_1,g_2)\) in \Cref{exp:D} satisfies
\(\mathcal P_{\mathbb Z}(g_1,g_2)=\varnothing\), so it already answers
\Cref{ques:Q65} negatively, as the empty polynomial family is jointly intersective.
If we want to find a `non-trivial' negative answer to \Cref{ques:Q65} where the jointly intersective family is non-empty, it is enough to adjoin an integer polynomial to the family, such as \(g_3(t)=t^2\):

\begin{example}[Negative answer to \Cref{ques:Q65}]\label{exa:E}
Let $\lambda\in\R\setminus\Q$, and choose $Q\in\N$ with $Q\normT{\sqrt2}>1+|\lambda|$.  Put
\[
h_1(t)=t^{3/2},\qquad h_2(t)=\lambda t^{3/2}+Q(t+\sqrt2),\qquad h_3(t)=t^2.
\]
Then
\[
\mathcal P_{\Z}(h_1,h_2,h_3)=\{ct^2:c\in\Z,\ c\ne0\},
\]
which is jointly intersective.  Furthermore, there exists a basic Bohr set \(E\subset\N\) such that
\[
R_{\round{h_1}}(E)\cap R_{\round{h_2}}(E)\cap R_{\round{h_3}}(E)=\varnothing.
\]
\end{example}

\section{Preliminaries}\label{sec:prelim}

\subsection{Basic Bohr sets}  

\begin{lemma}\label{lem:bohr-positive-density}
Let \(v_1,\dots,v_d\in\T\) and \(\varepsilon>0\). Put
\[
B(v_1,\dots,v_d;\varepsilon)
=
\{n\in\N:\normT{nv_j}<\varepsilon \text{ for }1\le j\le d\}.
\]
Then \(B(v_1,\dots,v_d;\varepsilon)\cap\N\) has positive natural density and is syndetic.
\end{lemma}

\begin{proof}
  Put \(\omega=(v_1,\dots,v_d)\in\T^d\), and let \(K=\overline{\{n\omega:n\in\Z\}}\subset\T^d\). As a compact Abelian group, \(K\) supports a unique Haar probability measure \(m_K\) for which the rotation \(x\mapsto x+\omega\) is uniquely ergodic. Let \(\pi_j\colon K\to\T\) be the \(j\)-th coordinate projection, whose image \(H_j = \pi_j(K)\) is a closed subgroup of \(\T\).

  Define \(V_j = \{y\in H_j : \normT{y}<\varepsilon\}\) and 
  \[
  U = \bigcap_{j=1}^d \pi_j^{-1}(V_j) = \{x\in K:\normT{x_j}<\varepsilon \text{ for }1\le j\le d\}.
  \]
  Since \(U\) is a non-empty relatively open neighborhood of \(0\) in \(K\), we have \(m_K(U)>0\). Note that \(B(v_1,\dots,v_d;\varepsilon)=\{n\in\N:n\omega\in U\}\).

  We next show that the relative boundary \(\partial_K U\) has \(m_K\)-measure zero. By elementary topology,
  \[
  \partial_K U \subset \bigcup_{j=1}^d \partial_K\bigl(\pi_j^{-1}(V_j)\bigr) \subset \bigcup_{j=1}^d \pi_j^{-1}\bigl(\partial_{H_j} V_j\bigr).
  \]
  The first inclusion holds because the boundary of a finite intersection lies in the union of the boundaries. The second follows from the continuity of \(\pi_j\), which guarantees \(\partial(f^{-1}(V)) \subset f^{-1}(\partial V)\) via standard properties of preimages of closures and interiors. Since the continuous, surjective group homomorphism \(\pi_j\) pushes \(m_K\) forward to the Haar measure \(m_{H_j}\), we have \(m_K\bigl(\pi_j^{-1}(\partial_{H_j} V_j)\bigr) = m_{H_j}(\partial_{H_j} V_j)\).

As a closed subgroup, \(H_j\) is either finite (making \(V_j\) clopen, so \(\partial_{H_j} V_j = \varnothing\)) or equal to \(\T\), in which case \(\partial_{\T} V_j\) is finite. In both cases, \(m_{H_j}(\partial_{H_j} V_j) = 0\), which forces \(m_K(\partial_K U) = 0\).

  Because \(m_K(\partial_K U) = 0\), the indicator function \(\mathbf{1}_U\) can be tightly approximated by continuous functions, that is, for any \(\delta>0\), there exist continuous functions \(f, g \in C(K)\) such that \(0 \le f \le \mathbf{1}_U \le g \le 1\) and \(\int_K g \, dm_K - \int_K f \, dm_K < \delta\). By unique ergodicity, the sequence of empirical measures \(\mu_N = \frac{1}{N}\sum_{n=1}^N \delta_{n\omega}\) converges weakly to \(m_K\). Applying this weak convergence to the continuous functions \(f\) and \(g\) yields
  \[
  \int_K f \, dm_K = \lim_{N\to\infty} \int_K f \, d\mu_N 
  \le \liminf_{N\to\infty} \mu_N(U) 
  \le \limsup_{N\to\infty} \mu_N(U) 
  \le \lim_{N\to\infty} \int_K g \, d\mu_N 
  = \int_K g \, dm_K.
  \]
   Since both integrals are within \(\delta\) of \(m_K(U)\) and \(\delta>0\) is arbitrary, the upper and lower limits coincide with \(m_K(U)\). Therefore,
 \[
 \lim_{N\to\infty}\frac{1}{N}
  |\{1\le n\le N:n\omega\in U\}|
  = \lim_{N\to\infty} \mu_N(U) = m_K(U) > 0.
  \]
   This proves that \(B(v_1,\dots,v_d;\varepsilon)\) has positive natural density. 
It remains to prove syndeticity. If \(B(v_1,\dots,v_d;\varepsilon)\) were not syndetic, then there would be intervals \(\{k_N+1,\dots,k_N+N\}\) disjoint from it. On the other hand, the same continuous approximation argument, together with the uniform convergence of Birkhoff averages in the uniquely ergodic rotation \(x\mapsto x+\omega\) on \(K\), gives
\[
\frac{1}{N}|\{1\le n\le N:(k_N+n)\omega\in U\}|\to m_K(U)>0,
\]
a contradiction.
\end{proof}

\subsection{Hardy fields}
A Hardy field is a field of germs at $+\infty$ of real-valued functions closed under differentiation.  All finite families below lie in a common Hardy field.  We use the following standard facts: every Hardy-field function is eventually monotone; if $f$ has polynomial growth, then $f^{(m)}(t)\to0$ for all sufficiently large $m$; and if a Hardy-field function has a finite limit, then its derivative tends to $0$.  See \cite[Section~2.2]{BMR24}.

\subsection{Finite-window representation of derivatives}

\begin{lemma}\label{lem:diff-iter}
Let $d\in\N$ and let $f:(0,\infty)\to\mathbb{R}$ be $C^d$ with $f^{(d)}(x)\to0$ as $x\to\infty$.  For each $1\le\ell\le d-1$ there exist rationals $b_{\ell,0},\ldots,b_{\ell,d-1}$, depending only on $d$ and $\ell$, such that
\[
f^{(\ell)}(x)=\sum_{j=0}^{d-1}b_{\ell,j}f(x+j)+o(1).
\]
\end{lemma}

\begin{proof}
For $h=0,\ldots,d-1$, Taylor's theorem gives
\[
f(x+h)=\sum_{p=0}^{d-1}\frac{h^p}{p!}f^{(p)}(x)+o(1),
\]
because $f^{(d)}(x)\to0$ and $h$ ranges over a fixed finite set.  Let
\[
\mathbf F(x)=(f(x),f(x+1),\ldots,f(x+d-1))^{\mathsf T},\qquad
\mathbf D(x)=(f(x),f'(x),\ldots,f^{(d-1)}(x))^{\mathsf T}.
\]
Then $\mathbf F(x)=VS\mathbf D(x)+o(1)$, where $V=(h^p)_{0\le h,p\le d-1}$ is a Vandermonde matrix and $S=\operatorname{diag}(1,1/1!,\ldots,1/(d-1)!)$.  Since $VS$ is invertible over $\Q$, the result follows by reading off the $\ell$-th component of $(VS)^{-1}\mathbf F(x)+o(1)$.
\end{proof}

The proof of \Cref{lem:diff-iter} implies that, for \(C^d\) functions \(f\) satisfying \(f^{(d)}\to 0\), we have
\begin{equation*}
\operatorname{span}_{\R}\left(f,f',f^{(2)},\dots,f^{(d-1)}\right)=_{\textup{mod }o(1)}
\operatorname{span}_{\R}\left(\{f(x+n);n\in\mathbb{Z}\}\right),
\end{equation*}
meaning that each function in the left-hand side differs by \(o(1)\) from a function in the right-hand side, and vice versa. The idea behind \Cref{DerivvsIterated1} of the following result is that, if we denote $\Delta f(x)=f(x+1)-f(x)$, then we also have
\begin{equation*}
\operatorname{span}_{\R}\left(f',f^{(2)},\dots,f^{(d-1)}\right)=_{\textup{mod }o(1)}
\operatorname{span}_{\R}\left(\{\Delta f(x+n);n\in\mathbb{Z}\}\right).
\end{equation*}
Similar statements hold for higher order derivatives and iterated differences, but we will not need them here.

\begin{lemma}[Window lemma]\label{lem:window}
Let $f_1,\ldots,f_k$ be of polynomial growth from a Hardy field.  Let $m\in\N$ be such that $f_i^{(m)}(t)\to0$ for every $i$.  Suppose
\[
g(t)=\sum_{i=1}^{k}\sum_{j=0}^{m-1}c_{i,j}f_i^{(j)}(t)\to L.
\]
Then the following hold.
\begin{enumerate}[label=\textup{(\alph*)}]
\item There exist real coefficients $\gamma_{i,q}$, $0\le q\le m-1$, such that
\[
\Psi(t)=\sum_{i=1}^{k}\sum_{q=0}^{m-1}\gamma_{i,q}f_i(t+q)=g(t)+o(1)\to L.
\]
\item\label{DerivvsIterated1} If the chosen representation has $c_{i,0}=0$ for every $i$, then there exist real coefficients $\widetilde\gamma_{i,q}$, $0\le q\le m-2$, such that
\[
\widetilde f(t)=\sum_{i=1}^{k}\sum_{q=0}^{m-2}\widetilde\gamma_{i,q}f_i(t+q)
\]
satisfies $\Delta\widetilde f(t)=g(t)+o(1)\to L$.
\end{enumerate}
\end{lemma}

\begin{proof}
Part (a) follows by substituting the expressions from \Cref{lem:diff-iter} into $g$ and collecting the coefficients of the shifts $f_i(t+q)$.

For (b), only positive-order derivatives occur.  For $j\ge1$, the coefficients $b_{j,q}$ in \Cref{lem:diff-iter} satisfy $\sum_{q=0}^{m-1}b_{j,q}=0$, as follows by applying the identity to a constant function.  Define $\beta_{j,q}=\sum_{s=q+1}^{m-1}b_{j,s}$ for $0\le q\le m-2$.  Then
\[
\Delta\left(\sum_{q=0}^{m-2}\beta_{j,q}f_i(t+q)\right)=\sum_{q=0}^{m-1}b_{j,q}f_i(t+q).
\]
Summing with coefficients $c_{i,j}$ gives the required $\widetilde f$.
\end{proof}

\section{Proof of the main results}

The following notation will be convenient in our arguments. For $L>0$ and $x\in\mathbb{R}$, denote 
\begin{equation*}
\left\|x\right\|_{\mathbb{R}/L\mathbb{Z}}:=\min\{|Ln-x|;n\in\mathbb{Z}\}
=L\left\|\frac{x}{L}\right\|_{\mathbb{T}}.
\end{equation*}
Note that $\|\cdot\|_{\mathbb{R}/L\mathbb{Z}}$ satisfies the triangle inequality, in the sense that for all $x,y,z\in\mathbb{R}$,
\begin{equation*}
\left\|x-z\right\|_{\mathbb{R}/L\mathbb{Z}}\leq
\left\|x-y\right\|_{\mathbb{R}/L\mathbb{Z}}
+
\left\|y-z\right\|_{\mathbb{R}/L\mathbb{Z}}.
\end{equation*}

During this section, let \(\rho:\R\to\Z\) be a rounding map, and put $K_\rho:=\sup_{x\in\R}|x-\rho(x)|<\infty .$

\begin{proposition}
\label{CoreEstimate}
Let $L>0,N,a_1,\dots,a_k,\lambda_1,\dots,\lambda_k\in\mathbb{R}$. Then
\begin{equation*}
\left\|\sum_{i=1}^k\lambda_i\rho( a_i)\right\|_{\mathbb{R}/L\mathbb{Z}}
\geq \|N\|_{\mathbb{R}/L\mathbb{Z}}-\left(
K_\rho\sum_{i=1}^k|\lambda_i|+\left|N-\sum_{i=1}^k\lambda_ia_i\right|\right)
\end{equation*}
\end{proposition}

\begin{proof}
Consider the following equality.
\begin{equation*}
N=\left(N-\sum_{i=1}^k\lambda_ia_i\right)+\sum_{i=1}^k\lambda_i(a_i-\rho(a_i))+\sum_{i=1}^k\lambda_i\rho(a_i).
\end{equation*}
Passing to $\mathbb{R}/L\mathbb{Z}$ and using the triangle inequality and $\|x\|_{\mathbb{R}/L\mathbb{Z}}\leq|x|$, we conclude:
\begin{align*}
\|N\|_{\mathbb{R}/L\mathbb{Z}}&
\leq\left\|N-\sum_{i=1}^k\lambda_ia_i\right\|_{\mathbb{R}/L\mathbb{Z}}+\sum_{i=1}^k\left\|\lambda_i(a_i-\rho(a_i))\right\|_{\mathbb{R}/L\mathbb{Z}}+\left\|\sum_{i=1}^k\lambda_i\rho(a_i)\right\|_{\mathbb{R}/L\mathbb{Z}}\\
&\leq
\left|N-\sum_{i=1}^k\lambda_ia_i\right|+K_\rho\sum_{i=1}^k|\lambda_i|+\left\|\sum_{i=1}^k\lambda_i\rho(a_i)\right\|_{\mathbb{R}/L\mathbb{Z}}.\qedhere
\end{align*}
\end{proof}

It will now be convenient to consider basic Bohr sets in the integers, which we denote as 
\[
B_{\mathbb{Z}}(v_1,\dots,v_d;\varepsilon)
:=
\{n\in\mathbb{Z}\>:\normT{nv_j}<\varepsilon \text{ for }1\le j\le d\}.
\]

\begin{proposition}
\label{LimitBohrObstruction}
Given functions $f_1,\dots,f_k:(0,\infty)\to\mathbb{R}$, suppose that for some $\lambda_i\in\mathbb{R}$ 
the following limit exists and is positive:
\begin{equation*}
L:=\lim_{x\to\infty}\sum_i\lambda_if_i(x)\in(0,\infty).
\end{equation*}
If $K_\rho\sum_i|\lambda_i|<L$, then there is a basic Bohr set $B\subset\Z$ such that, for big enough $x$, we do not have $\rho(f_i(x))\in B$ for all $i=1,\dots,k$. In particular, for any $0<\varepsilon<L-K_\rho\sum_i|\lambda_i|$ we can let 
\begin{equation*}
B=\left\{n\in\Z;\|\lambda_in\|_{\mathbb{R}/2L\mathbb{Z}}<\frac{\varepsilon}{k}\textup{ for all }i=1,\dots,k\right\}.
\end{equation*}
\end{proposition}

\begin{proof}
For big enough $x$, it follows from \Cref{CoreEstimate} that
\begin{equation*}
\sum_{i=1}^k\|\lambda_i\rho(f_i(x))\|_{\mathbb{R}/2L\mathbb{Z}}
\geq
\left\|\sum_{i=1}^k\lambda_i\rho(f_i(x))\right\|_{\mathbb{R}/2L\mathbb{Z}}
\geq
L-K_\rho\sum_i|\lambda_i|
-
\left|L-\sum_i\lambda_if_i(x)\right|
>\varepsilon.
\end{equation*}
This is a contradiction if $\|\lambda_i\rho(f_i(x))\|_{\mathbb{R}/2L\mathbb{Z}}<\frac{\varepsilon}{k}$ for all $i$, so we are done.
\end{proof}

\begin{corollary}
\label{LimitBohrObstructionWithTranslates}
Given functions \(f_1,\dots,f_k:(0,\infty)\to\mathbb R\) and \(m\in\mathbb N\), suppose that for some constants \(\lambda_{i,j}\in\mathbb R\), \(1\le i\le k\) and \(1\le j\le m\), the following limit exists and is positive:
\[
L:=\lim_{x\to\infty}\sum_{\substack{1\le i\le k\\1\le j\le m}}
\lambda_{i,j}f_i(x+j)\in(0,\infty).
\]
If \(K_\rho\sum_{i,j}|\lambda_{i,j}|<L\), then there exists a basic Bohr set \(E\subset\N\) such that \(\bigcap_{i=1}^k R_{\rho\circ f_i}(E)\) is not thick.
\end{corollary}

\begin{proof}
Apply \Cref{LimitBohrObstruction} to the finite family of translates \(f_i(\cdot+j)\),
\(1\le i\le k\), \(1\le j\le m\). We obtain a basic Bohr set
\(B_0=B_{\mathbb{Z}}(v_1,\dots,v_d;\varepsilon)\) such that, for all sufficiently large \(x\), the shifts
\(\rho(f_i(x+j))\), \(1\le i\le k\), \(1\le j\le m\), cannot all belong to \(B_0\).

Put \(E=B(v_1,\dots,v_d;\varepsilon/2)\). Then every difference
of two elements of \(E\) belongs to \(B_0\). If \(\bigcap_{i=1}^kR_{\rho\circ f_i}(E)\) were thick, then it would contain arbitrarily long intervals. Taking a sufficiently long such interval and then a subinterval with large initial point, we could find \(x\) large such that \(x+j\in \bigcap_{i=1}^kR_{\rho\circ f_i}(E)\) for \(1\le j\le m\). Hence \(\rho(f_i(x+j))\) is a difference of two elements of \(E\), and so lies in \(B_0\), for all \(i,j\). This is a contradiction.
\end{proof}

\begin{proof}[Proof of \Cref{thm:A}]
Suppose that some \(g\in\nablar(f_1',\dots,f_k')\) satisfies \(\lim_{t\to\infty}|g(t)|=L\in(0,\infty)\). Replacing \(g\) by \(-g\) if necessary, we may assume that \(g(t)\to L\). Let \(m\ge2\) be such that \(f_i^{(m)}(t)\to0\) for all \(i\). Then by \Cref{lem:window}\ref{DerivvsIterated1} there exist real coefficients $\widetilde\gamma_{i,q}$, $0\le q\le m-2$, such that
\[
\widetilde f(t):=\sum_{i=1}^{k}\sum_{q=0}^{m-2}\widetilde\gamma_{i,q}f_i(t+q)
\]
satisfies $\Delta\widetilde f(t)\to L$. Therefore, for all $M\in\mathbb{N}$, we have
\begin{equation*}
ML=\lim_{t\to\infty}\widetilde f(t+M)-\widetilde f(t)
=
\lim_{t\to\infty}
\left(\sum_{i=1}^{k}\sum_{q=0}^{m-2}\widetilde\gamma_{i,q}f_i(t+M+q)-\sum_{i=1}^{k}\sum_{q=0}^{m-2}\widetilde\gamma_{i,q}f_i(t+q)\right).
\end{equation*}
Therefore, choosing a constant $M\in\mathbb{N}$ such that $K_\rho\cdot2\sum_{i=1}^{k}\sum_{q=0}^{m-2}\left|\widetilde\gamma_{i,q}\right|<ML$, \Cref{LimitBohrObstructionWithTranslates} implies that there exists a basic Bohr set \(E\subset\N\) such that $\bigcap_{i=1}^k R_{\rho\circ f_i}(E)$ is not thick.
\end{proof}

\begin{proof}[Proof of \Cref{thm:B}]
If some $g\in\nablar(f_1,\ldots,f_k)$ satisfies $\lim_{t\to\infty}|g(t)|\in(0,\infty)$, then, replacing \(g\) by \(-g\) if necessary, we may assume that \(g(t)\to L\in(0,\infty)\). Let \(m\in\mathbb N\) be such that \(f_i^{(m)}(t)\to0\) for all \(i\), and such that \(g\) is a linear combination of \(f_i^{(j)}\), \(0\le j\le m-1\). Then by \Cref{lem:window} there are real coefficients $\gamma_{i,q}$, $0\le q\le m-1$, such that
\[
\sum_{i=1}^{k}\sum_{q=0}^{m-1}\gamma_{i,q}f_i(t+q)\stackrel{t\to\infty}\longrightarrow L.
\]
So for all $M>0$,
\[
\sum_{i=1}^{k}\sum_{q=0}^{m-1}\gamma_{i,q}\cdot(Mf_i)(t+q)\stackrel{t\to\infty}\longrightarrow ML.
\]
Choosing a constant $M\in\mathbb{N}$ such that $K_\rho\cdot\sum_{i=1}^{k}\sum_{q=0}^{m-1}\left|\gamma_{i,q}\right|<ML$, \Cref{LimitBohrObstructionWithTranslates} implies that there exists a basic Bohr set \(E\subset\N\) such that $\bigcap_{i=1}^k R_{\rho\circ(Mf_i)}(E)$ is not thick.
\end{proof}

\begin{proof}[Proof of \Cref{Thm:C}]
Let \(0<\varepsilon<\inf_{x\geq N}\left\|\sum_i\lambda_if_i(x)\right\|_{\mathbb T}-K_\rho\sum_i|\lambda_i|\), and define the basic Bohr set
\[
E:=\left\{n\in\N:
\|\lambda_j n\|_{\mathbb T}<\frac{\varepsilon}{2k}
\text{ for }1\le j\le k\right\}.
\]
For all \(x\geq N\), by the triangle inequality we have
\begin{align*}
\sum_{i=1}^k\|\lambda_i\rho(f_i(x))\|_{\mathbb T}
&\geq
\left\|\sum_{i=1}^k\lambda_i\rho(f_i(x))\right\|_{\mathbb T} \\
&\geq
\left\|\sum_{i=1}^k\lambda_i f_i(x)\right\|_{\mathbb T}
-K_\rho\sum_{i=1}^k|\lambda_i|>\varepsilon.
\end{align*}
If \(x\in\bigcap_{i=1}^k R_{\rho\circ f_i}(E)\), then
\(\rho(f_i(x))\in E-E\) for every \(i\). Hence
\(\|\lambda_i\rho(f_i(x))\|_{\mathbb T}<\varepsilon/k\) for every \(i\),
contradicting the preceding inequality. Therefore
\(\bigcap_{i=1}^k R_{\rho\circ f_i}(E)\cap [N,\infty)=\varnothing\).
\end{proof}

\begin{proof}[Proof of \Cref{Thm:CHomothety}]
Apply \Cref{Thm:C} to the functions $Mf_1,\dots,Mf_k$, using that for big enough $M$, we have
\begin{equation}
\inf_{x\geq N}\left\|\sum_{i}\lambda_if_i(x)\right\|_{\mathbb{T}}
=
\inf_{x\geq N}\left\|\sum_{i}\frac{\lambda_i}{M}\cdot Mf_i(x)\right\|_{\mathbb{T}}
> K_\rho\sum_i\left|\frac{\lambda_i}{M}\right|.\qedhere
\end{equation}
\end{proof}

\section{Examples}
\label{sec:examples}
\begin{proof}[Proof of \Cref{exa:C}]
Let $h_1(t)=t^{3/2}$ and $h_2(t)=\lambda t^{3/2}+t$.

Any function $F\in\nablaz(h_1,h_2)$ can be expressed, for some $M\in\mathbb{N}$ and $a_m,b_m\in\mathbb{Z}$, as
\[
F(t)=\sum_{m=0}^{M}a_mh_1^{(m)}(t)+\sum_{m=0}^{M}b_mh_2^{(m)}(t)
=\sum_{m=0}^{M}(a_m+\lambda b_m)c_mt^{3/2-m}+b_0t+b_1.
\]
If all pairs $(a_m,b_m)$ vanish, then $F=0$.  Otherwise let $m_0$ be the least index with $(a_{m_0},b_{m_0})\ne(0,0)$.  Since $\lambda\notin\Q$, $a_{m_0}+\lambda b_{m_0}\ne0$.  If $m_0\le1$, the corresponding term forces $|F(t)|\to\infty$.  If $m_0\ge2$, then $b_0=b_1=0$, and all remaining powers tend to $0$. Hence $\lim_{t\to\infty}|F(t)|\in\{0,\infty\}$.

We now construct $E$ for the floor convention. As $h_2'(t)-\lambda h_1'(t)=1$, for all $M>0$ we have
\begin{equation}
\label{RandomEq2we99}
(h_2(t+M)-\lambda h_1(t+M))-(h_2(t)-\lambda h_1(t))=M.
\end{equation}

Following the idea of the proof of \Cref{LimitBohrObstructionWithTranslates}, let \(M\geq100(1+|\lambda|)\), \(R=|\lambda|+10\), and
\begin{equation*}
E=\left\{n\in\N;\|n\|_{\mathbb{R}/2M\mathbb{Z}},\|\lambda n\|_{\mathbb{R}/2M\mathbb{Z}}<R\right\}.
\end{equation*}
Then $R_{\lfloor h_1\rfloor}(E)\cap R_{\lfloor h_2\rfloor}(E)$ is not thick: if so, some $n\in\mathbb{N}$ would satisfy $n,n+M\in R_{\lfloor h_1\rfloor}(E)\cap R_{\lfloor h_2\rfloor}(E)$, so 
\begin{equation*}
\|\lfloor h_2(n+M)\rfloor\|_{\mathbb{R}/2M\mathbb{Z}},
\|\lambda\lfloor h_1(n+M)\rfloor\|_{\mathbb{R}/2M\mathbb{Z}},
\|\lfloor h_2(n)\rfloor\|_{\mathbb{R}/2M\mathbb{Z}},
\|\lambda\lfloor h_1(n)\rfloor\|_{\mathbb{R}/2M\mathbb{Z}}
< 2R,
\end{equation*}
so
\begin{equation*}
\|h_2(n+M)\|_{\mathbb{R}/2M\mathbb{Z}},
\|\lambda h_1(n+M)\|_{\mathbb{R}/2M\mathbb{Z}},
\|h_2(n)\|_{\mathbb{R}/2M\mathbb{Z}},
\|\lambda h_1(n)\|_{\mathbb{R}/2M\mathbb{Z}}
<2R+1+|\lambda|,
\end{equation*}
so applying the triangle inequality to \Cref{RandomEq2we99} we would obtain
\(\|M\|_{\mathbb{R}/2M\mathbb{Z}}<8R+4+4|\lambda|\), a contradiction to
\(M\ge100(1+|\lambda|)\).

Finally, we check that $R_{\lfloor h_1\rfloor}(E)\cap R_{\lfloor h_2\rfloor}(E)$ is piecewise syndetic. Let
\begin{align*}
A&=\{n\in\mathbb{N};\|\lfloor n^{3/2}\rfloor\|_{\mathbb{R}/2M\mathbb{Z}},\|\lfloor\lambda n^{3/2}\rfloor\|_{\mathbb{R}/2M\mathbb{Z}},\|\lfloor\lambda^2 n^{3/2}\rfloor\|_{\mathbb{R}/2M\mathbb{Z}}<1\},\\
B&=\{n\in\mathbb{N};\| n\|_{\mathbb{R}/2M\mathbb{Z}},\| \lambda n\|_{\mathbb{R}/2M\mathbb{Z}}<1\}.
\end{align*}
Then $A$ is thick by applying \Cref{thm:BMR-thick} to the positive-density set $2M\mathbb N$ and the functions $t^{3/2},\lambda t^{3/2},\lambda^2t^{3/2}$, and $B$ is syndetic by \Cref{lem:bohr-positive-density}. Hence \(A\cap B\) is piecewise syndetic. So it will be enough to prove that $A\cap B\subset R_{\lfloor h_1\rfloor}(E)\cap R_{\lfloor h_2\rfloor}(E)$.
 
For \(n\in A\cap B\), put \(u_1=\lfloor h_1(n)\rfloor\) and \(u_2=\lfloor h_2(n)\rfloor\). Then \(\|u_1\|_{\mathbb R/2M\mathbb Z}<1\) and
\(\|\lambda u_1\|_{\mathbb R/2M\mathbb Z}<|\lambda|+2\). Also,
\[
\|u_2\|_{\mathbb R/2M\mathbb Z}<3.
\]
It remains only to bound \(\|\lambda u_2\|_{\mathbb R/2M\mathbb Z}\). Using the fact that
\(|\lfloor x+y\rfloor-(\lfloor x\rfloor+\lfloor y\rfloor)|\leq1\), for all \(n\in A\cap B\) we have
\begin{multline*}
\|\lfloor\lambda h_2(n)\rfloor\|_{\mathbb{R}/2M\mathbb{Z}}
=
\|\lfloor\lambda^2n^{3/2}+\lambda n\rfloor\|_{\mathbb{R}/2M\mathbb{Z}}\\\leq 
\|\lfloor\lambda^2n^{3/2}\rfloor\|_{\mathbb{R}/2M\mathbb{Z}}+\|\lfloor\lambda n\rfloor\|_{\mathbb{R}/2M\mathbb{Z}}+1
<3.
\end{multline*}
Hence \(\|\lambda u_2\|_{\mathbb R/2M\mathbb Z}<|\lambda|+4\). Thus \(u_1,u_2\) both lie in the smaller Bohr neighborhood
\[
\left\{r\in\mathbb Z:\|r\|_{\mathbb R/2M\mathbb Z},\|\lambda r\|_{\mathbb R/2M\mathbb Z}<|\lambda|+5\right\}.
\]
Since \(R=|\lambda|+10\), each such \(u_i\) is a return time for $E$: choose a sufficiently large positive element \(a\) in the Bohr set with radius \(R-(|\lambda|+5)\), then \(a,a+u_i\in E\).

 The ceiling and nearest-integer conventions are obtained by repeating the same proof with \(\rho=\lceil\cdot\rceil\) or \(\rho=[\cdot]\) in place of \(\lfloor\cdot\rfloor\). The estimates above only use \(|\rho(x)-x|\leq1\) and \(|\rho(x+y)-\rho(x)-\rho(y)|\leq2\), so the constants \(10\) and \(100\) absorb the resulting changes.
\end{proof}

\begin{proof}[Proof of \Cref{exp:D}]
The integer derivative-span assertion is proved exactly as in \Cref{exa:C}, with the linear term $t$ replaced by $Q(t+\xi)$. Since
\[
-\frac{\lambda}{Q}g_1(t)+\frac{1}{Q}g_2(t)=t+\xi,
\]
and \(\inf_{x\in\mathbb N}\|x+\xi\|_{\mathbb T}=\|\xi\|_{\mathbb T}>(1+|\lambda|)/Q\), \Cref{Thm:C} applies to the coefficients \(-\lambda/Q\) and \(1/Q\). Thus, for each of the choices \(\rho=\floor{\cdot}\), \(\rho=\ceil{\cdot}\), and \(\rho=\round{\cdot}\), there exists a basic Bohr set \(E\subset\N\) such that
\(R_{\rho\circ g_1}(E)\cap R_{\rho\circ g_2}(E)=\varnothing\).
\end{proof}

\begin{proof}[Proof of \Cref{exa:E}]
Suppose $q\in\mathcal P_{\Z}(h_1,h_2,h_3)$.  Then there exist $a,b,c\in\R$, not all zero, such that
\[
(a+\lambda b)t^{3/2}+ct^2+bQ(t+\sqrt2)-q(t)\to0.
\]
Since $t^{3/2}$ is not $o(1)$-close to any polynomial, $a+\lambda b=0$.  Hence
\[
q(t)=ct^2+bQt+bQ\sqrt2.
\]
As $q\in\Z[t]$, the coefficients $c$, $bQ$, and $bQ\sqrt2$ must be integers. As $bQ$ and $bQ\sqrt2$ are integers we have $Qb=0$, so $b=0$ and $a=0$. Therefore $q(t)=ct^2$ with $c\in\Z\setminus\{0\}$.  Conversely, every nonzero integer multiple of $t^2$ is obtained from $h_3$.  Thus
\[
\mathcal P_{\Z}(h_1,h_2,h_3)=\{ct^2:c\in\Z,\ c\ne0\}.
\]
This family is jointly intersective: for every modulus $m$, taking $n=m$ gives $cm^2\equiv0\pmod m$ for every $c\in\Z$. 

The empty-return assertion follows from \Cref{exp:D}: Since the pairwise intersection is empty, the triple intersection is empty as well.
\end{proof}

\begin{proposition}[An intrinsic sufficient condition for the BMR polynomial hypothesis]\label{prop:subspace-sufficiency}
Retain the notation of \Cref{thm:BMR-A}. Let \(\mathcal P_{\R}=\operatorname{poly}(f_1,\ldots,f_k)\) and \(P_{\Z}=(\mathcal P_{\R}\cap\Z[t])\setminus\{0\}\). If \(P_{\Z}\) is jointly intersective and
\[
\mathcal P_{\R}=\operatorname{span}_{\R}P_{\Z},
\]
then the strong polynomial-span hypothesis of \Cref{thm:BMR-A} holds.
\end{proposition}

\begin{proof}
Since $\operatorname{span}_{\R}(f_1,\ldots,f_k)$ is finite-dimensional, so is $\mathcal P_{\R}$. If $\mathcal P_{\R}=\{0\}$, the BMR polynomial-span hypothesis is immediate, for instance by taking the jointly intersective polynomial $q(t)=t$. Otherwise, choose a finite basis $q_1,\ldots,q_\ell$ of $\mathcal P_{\R}$ from $P_{\Z}$. This finite subfamily is jointly intersective because $P_{\Z}$ is. Hence
\[
\mathcal P_{\R}\subset\operatorname{span}_{\R}(q_1,\ldots,q_\ell),
\]
which is exactly the BMR polynomial-span hypothesis.
\end{proof}

\begin{remark}[Shadow-space decoupling]
For the triple in \Cref{exa:E}, the real polynomial-shadow space $\mathcal P_{\R}$ contains both
\[
Q\left(t+\sqrt2\right)=h_2-\lambda h_1\textup{ and }
 t^2=h_3.
\]
However, the integer shadow only sees the $t^2$-axis:
\[
\mathcal P_{\Z}(h_1,h_2,h_3)=\{ct^2:c\in\Z,\ c\ne0\}.
\]
Thus, the jointly intersective integer shadow is supplied purely by the visible polynomial direction \(t^2\). In the language of \Cref{prop:subspace-sufficiency}, we have the strict inclusion
\[
\operatorname{span}_{\R}\mathcal P_{\Z}(h_1,h_2,h_3)\subsetneq \mathcal P_{\R}(h_1,h_2,h_3).
\]The Bohr obstruction arises precisely from the real direction $Q(t+\sqrt2)$ living in this complementary gap, which remains invisible to $P_{\Z}$ because of its irrational constant term.
\end{remark}

\section*{Acknowledgments}
The authors would like to thank Professors Wen Huang, Piotr Oprocha, Song Shao, and Xiangdong Ye for helpful discussions and valuable suggestions. We also thank Joel Moreira and Vitaly Bergelson for bringing to Saúl Rodríguez's attention the recent work of Kangbo Ouyang, Leiye Xu, and Shuhao Zhang on Bohr obstructions to recurrence along Hardy-field sequences, which led to the present collaboration.

\bigskip
\section*{Addresses}

\noindent\textsc{Kangbo Ouyang}\\
School of Mathematical Sciences, University of Science and Technology of China,\\
Hefei, Anhui 230026, P.R. China\\
\textit{Email address}: \texttt{oy19981231@mail.ustc.edu.cn}

\medskip
\noindent\textsc{Sa\'ul Rodr\'iguez Mart\'in}\\
Department of Mathematics, The Ohio State University,\\
100 Math Tower, 231 West 18th Avenue,\\
Columbus, OH 43210-1174, USA\\
\textit{Email address}: \texttt{rodriguezmartin.1@osu.edu}

\medskip
\noindent\textsc{Leiye Xu}\\
School of Mathematical Sciences, University of Science and Technology of China,\\
Hefei, Anhui 230026, P.R. China\\
\textit{Email address}: \texttt{leoasa@mail.ustc.edu.cn}

\medskip
\noindent\textsc{Shuhao Zhang}\\
School of Mathematical Sciences, University of Science and Technology of China,\\
Hefei, Anhui 230026, P.R. China\\
\textit{Email address}: \texttt{yichen12@mail.ustc.edu.cn}

\end{document}